\documentclass[12pt]{amsart}

\usepackage{amssymb,latexsym,amscd}
\usepackage{amsthm,amsmath,mathrsfs}

\newcommand{\Mfr}{\mathfrak{M}}

\newcommand{\cV}{\mathcal{V}}
\newcommand{\cH}{\mathcal{H}}

\newcommand{\cZ}{\mathcal{Z}}

\newcommand{\cO}{\mathcal{O}}

\newcommand{\cF}{\mathcal{F}}

\newcommand{\cB}{\mathcal{B}}
\newcommand{\cD}{\mathcal{D}}

\newcommand{\cL}{\mathcal{L}}

\newcommand{\cC}{\mathcal{C}}

\newcommand{\cS}{\mathcal{S}}

\newcommand{\DD}{\mathrm{D}}

\newcommand{\Pic}{\mathrm{Pic}}

\newcommand{\Gal}{\mathrm{Gal}}

\newcommand{\lra}{\longrightarrow}

\newcommand{\ra}{\rightarrow}

\newcommand{\PP}{\mathbb{P}}
\newcommand{\AAA}{\mathbb{A}}

\newcommand{\ZZ}{\mathbb{Z}}

\newcommand{\QQ}{\mathbb{Q}}
\newcommand{\NN}{\mathbb{N}}
\newcommand{\RR}{\mathbb{R}}

\newcommand{\CC}{\mathbb{C}}

\newcommand{\GL}{\mathrm{GL}}
\newcommand{\PGL}{\mathbb{P}\mathrm{GL}}
\newcommand{\SL}{\mathrm{SL}}

\newcommand{\slfr}{\mathfrak{sl}}
\newcommand{\Mat}{\mathrm{Mat}}
\newcommand{\ord}{\mathrm{ord}}
\def\map#1{\ \smash{\mathop{\longrightarrow}\limits^{#1}}\ }

\theoremstyle{plain}

\newtheorem{thm}{Theorem}[section]

\newtheorem{lem}[thm]{Lemma}

\newtheorem{prop}[thm]{Proposition}

\newtheorem{cor}[thm]{Corollary}

\newtheorem{rem}[thm]{Remark}

\let\mathbb=\mathbf

\begin{document}

\title[]{On the monodromy of the Hitchin connection}

\begin{abstract}
For any genus $g \geq 2$ we give an example of a family of smooth complex projective curves of genus $g$ such that 
the image of the monodromy representation of the Hitchin connection
on the sheaf of generalized $\SL(2)$-theta functions of level $l \not= 1,2,4$ and $8$ contains an 
element of infinite order.
\end{abstract}

\author{Yves Laszlo}

\author{Christian Pauly}

\author{Christoph Sorger}

\address{D\'epartement de Math\'ematiques B\^at. 425 \\
Universit\'e Paris-Sud \\
91405 Orsay Cedex \\
France}

\email{yves.laszlo@math.u-psud.fr}

\address{Laboratoire de Math\'{e}matiques J.-A. Dieudonn\'e \\ Universit\'e de Nice - Sophia Antipolis \\
06108 Nice Cedex 02 \\ France}

\email{pauly@unice.fr}

\address{Laboratoire de Math\'{e}matiques Jean Leray \\ Universit\'{e} de Nantes \\ 2, rue de la Houssini\`{e}re \\ BP 92208 \\
44322 Nantes Cedex 03 \\ France}

\email{christoph.sorger@univ-nantes.fr}

\thanks{Partially supported by ANR grant G-FIB}



\maketitle

\bigskip

\section{Introduction}

Let $\pi:\cC\ra\cB$ be a family of smooth connected complex projective curves
of genus $g\geq 2$ parameterized by a smooth complex manifold $\cB$. 
For any integers $l \geq 1$, called the level, and $r \geq 2$ we denote $\cZ_l$ the
complex vector bundle over $\cB$ having fibers
$H^0(\mathrm{M}_{\cC_b}(\SL(r)), \cL^{\otimes l})$, where
$\mathrm{M}_{\cC_b}(\SL(r))$ is the moduli space of semistable rank-$r$
vector bundles with trivial determinant over the curve $\cC_b=\pi^{-1}(b)$ for $b\in\cB$
and $\cL$ is the ample generator of its Picard group.
Following Hitchin \cite{H}, the bundle $\cZ_l$ is equipped with a
projectively flat connection called the Hitchin connection.

\bigskip

The main result of this paper is the following

\bigskip

{\bf Theorem.} {\em Assume that the level $l \not= 1,2,4$ and $8$ and that the rank $r=2$. 
For any genus $g \geq 2$ there exists a 
family $\pi: \cC \ra \cB$
of smooth complex connected projective curves of genus $g$ such that the monodromy representation of the 
Hitchin connection
$$ \rho_l : \pi_1(\cB,b) \lra \PP \GL(\cZ_{l,b})$$
has an element of infinite order in its image.}

\bigskip

For any genus $g \geq 2$ we give an example of a family  $\pi: \cC \ra \cB$ of smooth hyperelliptic 
curves of genus $g$ and an explicit
element $\xi \in \pi_1(\cB,b)$ with image of infinite order (see Remark \ref{explicitfamily}).

\bigskip
In the context of Witten-Reshetikhin-Turaev Topological Quantum Field Theory as defined by Blanchet-Habegger-Masbaum-Vogel \cite{BHMV}, the analogue of the above theorem is well-known due to work of Masbaum \cite{Mas}, who exhibited an explicit element of the mapping class group with image of infinite order. Previously, Funar \cite{F} had shown by a
different argument the somewhat weaker result that the image of the mapping class group is an infinite group.

It is enough to show the above theorem in the context of Conformal Field Theory as defined by Tsuchiya-Ueno-Yamada \cite{TUY}: following a result of the first author \cite{L}, the monodromy representation associated to Hitchin's connection coincides with the monodromy representation of the WZW connection. In a series of papers by
Andersen and Ueno (\cite{AU1}, \cite{AU2}, \cite{AU3} and \cite{AU4}) it has been shown recently that
the above Conformal Field Theory and the above Topological Quantum Field Theory are equivalent.
Therefore the above theorem also follows from that identification and the work of Funar and Masbaum.

In this short note, we give a direct algebraic proof, avoiding the above identification: we first recall Masbaum's initial argument applied to Tsuchiya-Kanie's description of the monodromy representation for the WZW connection in the case of the projective line with 4 marked points (see also \cite{AMU}). Then we observe that the sewing procedure induces a projectively flat map between sheaves of conformal blocks, enabling us to increase the genus of the curve.

\bigskip

A couple of words about the exceptional levels $l=1,2,4,8$ are in order.
For $l=1$ the monodromy representation $\rho_1$ is finite for any $g$. This follows from the fact that the Beauville-Narasimhan-Ramanan \cite{BNR} strange duality isomorphism 
$\PP H^0(\mathrm{M}_{\cC_b}(\SL(2)), \cL)^\dagger \stackrel{\sim}{\rightarrow}
\PP  H^0(\Pic^{g-1}(\cC_b), 2 \Theta)$ is projectively flat over $\cB$ for any family $\pi : \cC \ra \cB$ (see e.g. \cite{Bel1}) and that
$\rho_1$ thus identifies with the monodromy representation on a space of abelian theta functions, which is known to have finite image (see e.g. \cite{W}).
For $l=2$ there is a canonical morphism $H^0(\mathrm{M}_{\cC_b}(\SL(2)), \cL^{\otimes 2}){\rightarrow}
H^0(\Pic^{g-1}(\cC_b), 4\Theta)_+$, which is an isomorphism if and only if $\cC_b$ has no vanishing theta-null \cite{B2}. But this map is not projectively flat having non-constant rank. So the
question about finiteness of $\rho_2$ remains open --- see also \cite{Bel2}.
For $l=4$ there is a canonical isomorphism \cite{OP}, \cite{AM} between the dual 
$H^0(\mathrm{M}_{\cC_b}(\SL(2)), \cL^{\otimes 4})^\dagger$ and a space
of abelian theta functions of order $3$. We expect this isomorphism to be projectively flat. For $l=8$ no 
isomorphism with spaces of abelian theta functions seems to be known.

\bigskip

Our motivation to study the monodromy representation of the Hitchin connection comes from the Grothendieck-Katz conjectures
on the $p$-curvatures of a local system \cite{K}. In a forthcoming paper we will discuss the  consequences of the above theorem in this set-up.

\bigskip

{\bf Acknowledgements:}  We would like to thank Jean-Beno\^it Bost, Louis Funar and Gregor
Masbaum for helpful conversations and an anonymous referee for useful remarks on a first version of this paper.

\bigskip

\section{Review of mapping class groups, moduli spaces of  pointed curves  and braid groups}

\subsection{Mapping class groups}

In this section we recall the basic definitions and properties of the mapping class groups. We
refer the reader e.g. to \cite{I} or \cite{HL}.

\subsubsection{Definitions}

Let $S$ be a compact oriented surface of genus $g$ without boundary and with $n$ marked points $x_1, \ldots , x_n \in S$.
Associated to the $n$-pointed surface $S$ are the mapping class groups $\Gamma_g^n$ and $\Gamma_{g,n}$
defined as the groups of isotopy classes of orientation-preserving diffeomorphisms $\phi: S \ra S$ such that
$\phi(x_i) = x_i$ for each $i$, respectively such that $\phi(x_i) = x_i$ and the differential $d\phi_{x_i} : T_{x_i} S \ra
T_{x_i} S$  at the point $x_i$ is the identity map for each $i$.

\bigskip

An alternative definition of the mapping class groups $\Gamma_g^n$ and $\Gamma_{g,n}$ can be given
in terms of surfaces with boundary. We consider the surface $R$ obtained from $S$ by removing a small
disc around each marked point $x_i$. The boundary $\partial R$ consists
of $n$ circles.  Equivalently, the groups  $\Gamma_g^n$ and $\Gamma_{g,n}$  coincide with the
groups of isotopy classes of orientation-preserving diffeomorphisms $\phi: R \ra R$ such that
$\phi$ preserves each boundary component of $R$, respectively such that  $\phi$ is the identity on $\partial R$.

\bigskip

The mapping class group $\Gamma_g$ is defined to be $\Gamma_g^0 = \Gamma_{g,0}$.

\subsubsection{Dehn twists}

Given an (unparametrized) oriented, embedded circle $\gamma$ in $R \subset S$ we can associate to it a diffeomorphism $T_\gamma$ up
to isotopy, i.e., an element $T_\gamma$ in the mapping class groups $\Gamma_g^n$ and $\Gamma_{g,n}$, the so-called
Dehn twist along the curve $\gamma$. It is known that the mapping class groups  $\Gamma_g^n$ and $\Gamma_{g,n}$ are
generated by a finite number of Dehn twists. We recall the following exact sequence
$$ 1 \lra \ZZ^n \lra \Gamma_{g,n} \lra \Gamma_g^n \lra 1.$$
The $n$ generators of the abelian kernel $\ZZ^n$ are given by the Dehn twists $T_{\gamma_i}$,
where $\gamma_i$ is a loop going around
the boundary circle associated to $x_i$ for each $i$.

\subsubsection{The mapping class groups $\Gamma_0^4$ and $\Gamma_{0,4}$}

Because of their importance in this paper we recall the presentation of the mapping class
groups $\Gamma_0^4$ and $\Gamma_{0,4}$ by generators and relations. Keeping the notation of the previous section,
we denote by $R$ the $4$-holed sphere and by $\gamma_1, \gamma_2, \gamma_3, \gamma_4$ the circles in $R$ around the four
boundary circles. We denote by $\gamma_{ij}$ the circle dividing $R$ into two parts containing two holes each and such that
the two circles $\gamma_i$ and $\gamma_j$ are in the same part. It is known (see e.g. \cite{I} section 4) that $\Gamma_{0,4}$ is generated by the Dehn twists $T_{\gamma_i}$ for $1 \leq i \leq 4$ and $T_{\gamma_{ij}}$ for $1 \leq i,j \leq 3$
and that, given a suitable orientation of the circles $\gamma_i$ and $\gamma_{ij}$,
there is a relation  (the {\em lantern} relation)
$$ T_{\gamma_1} T_{\gamma_2} T_{\gamma_3} T_{\gamma_4} = T_{\gamma_{12}} T_{\gamma_{13}} T_{\gamma_{23}}.$$
Note that the images of the Dehn twists $T_{\gamma_i}$ under the natural homomorphism
$$ \Gamma_{0,4} \longrightarrow \Gamma_0^4 , \qquad T_\gamma \mapsto \overline{T}_\gamma, $$
are trivial. Thus the group
$\Gamma_0^4$ is generated by the three Dehn twists $\overline{T}_{ij}$ for $1 \leq i,j \leq 3$ with the relation
$\overline{T}_{\gamma_{12}} \overline{T}_{\gamma_{13}} \overline{T}_{\gamma_{23}} = 1$.

\bigskip

For each $4$-holed sphere being contained in a closed genus $g$ surface without
boundary one can consider the Dehn twists $T_{ij}$ as elements in the mapping class
group $\Gamma_g$.

\subsection{Moduli spaces of curves}

Let $\Mfr_{g,n}$ denote the moduli space parameterizing $n$-pointed smooth projective curves of genus $g$. The moduli
space $\Mfr_{g,n}$ is a (possibly singular) algebraic variety. It can also be thought of as an orbifold 
(or Deligne-Mumford stack) and one has an isomorphism
\begin{equation} \label{isomodmcg}
j: \pi_1 (\Mfr_{g,n},x) \map{\sim} \Gamma_g^n,
\end{equation}
where $\pi_1 (\Mfr_{g,n},x)$ stands for the orbifold fundamental group of $\Mfr_{g,n}$. In case the
space $\Mfr_{g,n}$ is a smooth algebraic variety, the orbifold fundamental group coincides with the usual fundamental group.

\subsection{The isomorphism between $\pi_1(\Mfr_{0,4},x)$ and $\Gamma_0^4$}

The moduli space $\Mfr_{0,4}$ parameterizes ordered sets of $4$ points on the complex projective line $\PP_\CC^1$ up to the diagonal action
of $\PP \GL(2, \CC)$. The cross-ratio induces an isomorphism with the projective line $\PP_\CC^1$ with $3$ punctures at $0,1$ and $\infty$
$$\Mfr_{0,4} \map{\sim} \PP^1_\CC \setminus \{ 0,1, \infty \}. $$
We deduce that the fundamental group  of $\Mfr_{0,4}$ is the group with three  generators
$$ \pi_1(\Mfr_{0,4}, x) = \langle \sigma_1 , \sigma_2 , \sigma_3 \ | \ \sigma_3 \sigma_2 \sigma_1 = 1 \rangle, $$
where $\sigma_1, \sigma_2$  and $\sigma_3$ are the loops starting at $x \in
\PP^1_\CC \setminus \{ 0,1, \infty \}$ and going once around the points $0,1$ and
$\infty$ with the same orientation. We choose the orientation such that the generators $\sigma_i$ satisfy the
relation $\sigma_3 \sigma_2 \sigma_1 = 1$. Clearly $\pi_1(\Mfr_{0,4}, x)$ coincides with the
fundamental group $\pi_1(Q,x)$ of the $3$-holed sphere $Q$.

\bigskip
In this particular case the isomorphism $j: \pi_1 (\Mfr_{0,4},x) \map{\sim} \Gamma_0^4$ can be explicitly described as follows
(see e.g. \cite{I} Theorem 2.8.C):
we may view the $3$-holed sphere $Q$ as the union of the $4$-holed sphere $R$ with a disc $D$ glued on the boundary corresponding to the point $x_4$. Given a loop $\sigma \in \pi_1(Q,x)$ we may find an isotopy $\{ f_t : Q \ra Q \}_{0 \leq t \leq 1}$ such that
the map $t \mapsto f_t(x)$ coincides with the loop $\sigma$, $f_0 = \mathrm{id}_Q$ and $f_1(D) = D$. Then the isotopy class of
$f_1$ resticted to $R \subset Q$ determines an element $j(\sigma) = [f_1] \in  \Gamma_0^4$. Moreover, with the
previous notation, we have the
equalities (see e.g. \cite{I} Lemma 4.1.I)
$$j(\sigma_1) = \overline{T}_{\gamma_{23}}, \qquad j(\sigma_2) = \overline{T}_{\gamma_{13}}, \qquad j(\sigma_3) =
\overline{T}_{\gamma_{12}}.$$

\begin{rem}
{\em At this stage we observe that under the isomorphism $j$ the two elements $\sigma_1^{-1} \sigma_2 \in \pi_1 (\Mfr_{0,4},x)$
and $\overline{T}_{\gamma_{23}}^{-1} \overline{T}_{\gamma_{13}} \in \Gamma^4_0$ coincide. It was shown by G. Masbaum in \cite{Mas} that the
latter element has infinite order in the TQFT-representation of the mapping class group $\Gamma_g$ --- note that
$T_{\gamma_{23}}^{-1} T_{\gamma_{13}}$ also makes sense in $\Gamma_g$. We will show in Proposition \ref{sigmainf} that the loop $\sigma_1^{-1} \sigma_2$
has infinite order in the monodromy representation of the WZW connection. }
\end{rem}

\subsection{Braid groups and configuration spaces}

We recall some basic results
about braid groups and configuration spaces. We refer the reader e.g. to \cite{KT} Chapter 1.

\subsubsection{Definitions} \label{definitionconfspace}

The braid group $B_n$ is the group generated by $n-1$ generators $g_1, \ldots, g_{n-1}$ and the
relations
$$ g_i g_{i+1} g_i = g_{i+1} g_i g_{i+1}, \ 1 \leq i \leq n-2, \qquad \text{and} \qquad g_ig_j = g_j g_i, \
|i-j| \geq 2.$$
The pure braid group is the kernel $P_n = \ker (B_n \ra \Sigma_n)$ of the group
homomorphism which associates to the generator $g_i$ the transposition $(i,i+1)$ in the symmetric group $\Sigma_n$.
The braid groups $B_n$ and $P_n$ can be identified with the
fundamental groups
$$ P_n = \pi_1(X_n, p), \qquad B_n = \pi_1(\overline{X}_n, \overline{p}), $$
where $X_n$ and $\overline{X}_n$ are the complex manifolds parameterizing ordered respectively unordered $n$-tuples
of distinct points in the complex plane
$$ X_n = \{ (z_1,z_2, \ldots,  z_n) \in \CC^n \ | \ z_i \not= z_j \} \qquad
\text{and} \qquad \overline{X}_n = X_n / \Sigma_n . $$
The points $p = (z_1, \ldots, z_n)$ and $\overline{p} = p \ \mathrm{mod} \ \Sigma_n$ are base 
points in $X_n$ and $\overline{X}_n$. There are natural
inclusions $B_n \hookrightarrow B_{n+1}$, which induce inclusions on the pure braid groups
$\iota : P_n \hookrightarrow P_{n+1}$.

Over the variety $X_n$ there is an universal family 
\begin{equation} \label{univfam}
\cF_{n +1} = (\pi :\cC = X_n \times \PP^1 \ra X_n; s_1, \ldots, s_n, s_\infty),
\end{equation}
parameterizing $n+1$ distinct points on the projective line $\PP^1$. The section
$s_i$ is given by the natural projection $X_n \ra \CC$ on the $i$-th component followed by the
inclusion $\CC \subset \PP^1_\CC = \CC \cup \{ \infty \}$ and $s_\infty$ is the constant section corresponding 
to $\infty \in \PP^1_\CC$.

\subsubsection{Relation between the pure braid group $P_3$ and the fundamental group
$\pi_1(\Mfr_{0,4},x)$} \label{M04}

The natural map
$$\Mfr_{0,4} = \PP^1_\CC \setminus \{ 0,1, \infty \} \lra X_3, \qquad z \mapsto (0,1,z)$$
induces a group homomorphism at the level of fundamental groups
$$\Psi : \pi_1(\Mfr_{0,4}, x) = \langle \sigma_1 , \sigma_2 \rangle \lra
P_3 = \pi_1(X_3, p_3),$$
with $p_3 = (0,1,x)$. Then $\Psi$ is a monomorphism by \cite{KT} Theorem 1.16. Moreover,
the image of $\Psi$ coincides with the kernel of the natural group homomorphism
$$ \mathrm{im} \ \Psi  = \ker \left( P_3 = \pi_1(X_3, p_3) \lra P_2 = \pi_1(X_2, p_2) \right) $$
induced by the projection onto the first two factors $X_3 \ra X_2$, $(z_1, z_2, z_3)
\mapsto (z_1, z_2)$ and $p_2 = (0,1)$. One computes explicitly (see \cite{KT} section 1.4.2)  that
$$ \Psi(\sigma_1) = g_2 g_1^2  g_2^{-1}, \qquad \text{and} \qquad  \Psi(\sigma_2) = g_2^2.$$
For later use we introduce the element
\begin{equation} \label{defsigma}
\sigma = \sigma_1^{-1} \sigma_2 \in \pi_1(\Mfr_{0,4}, x).
\end{equation}

\section{Conformal blocks and the projective WZW connection}

\subsection{General set-up}

We consider the simple Lie algebra $\slfr(2)$. The set of irreducible $\slfr(2)$-modules, i.e. the set of
dominant weights of $\slfr(2)$ equals 
$$ P_+ = \{ \lambda = m \varpi  \ | \ m \in \NN \}, $$
where $\varpi$ is the fundamental weight of $\slfr(2)$, which 
corresponds to the standard $2$-dimensional representation of $\slfr(2)$.
We fix an integer $l \geq 1$, called the level, and introduce the finite set 
$P_l = \{ \lambda \in P_+ \ | \  m \leq l \}$. Given any $\lambda \in P_l$
we denote by $\lambda^\dagger \in P_l$ the dominant weight of the dual $V_\lambda^\dagger$ of the 
$\slfr(2)$-module $V_\lambda$ with dominant weight $\lambda$. Note that $\lambda^\dagger = \lambda$. 
Given an integer $n \geq 1$, a
collection $\vec{\lambda} =  (\lambda_1, \ldots ,
\lambda_n) \in (P_l)^n$ of dominants weights of $\slfr(2)$ and a family
$$ \cF = ( \pi: \cC \ra \cB ; s_1, \ldots , s_n ; \xi_1 , \ldots , \xi_n) $$
of $n$-pointed stable curves of arithmetic genus $g$ parameterized by a base variety $\cB$ with sections $s_i : \cB \ra \cC$ and
formal coordinates $\xi_i$ at the divisor $s_i(\cB) \subset \cC$, one constructs (see \cite{TUY} section 4.1)
a locally free sheaf
$$\cV^\dagger_{l, \vec{\lambda}}(\cF)$$
over the base variety $\cB$, called
the {\em sheaf of conformal blocks} or the {\em sheaf of vacua}. We recall that
$\cV^\dagger_{l, \vec{\lambda}}(\cF)$ is a subsheaf of $\cO_{\cB} \otimes
\cH^\dagger_{\vec{\lambda}}$, where $\cH^\dagger_{\vec{\lambda}}$ denotes the dual
of the tensor product $\cH_{\vec{\lambda}} = \cH_{\lambda_1} \otimes \cdots \otimes
\cH_{\lambda_n}$ of the integrable highest weight representations $\cH_{\lambda_i}$ of
level $l$ and weight $\lambda_i$ of the affine Lie algebra $\widehat{\slfr(2)}$.
The formation of the sheaf of
conformal blocks commutes with base change. In particular, we have for any point $b \in \cB$
$$  \cV^\dagger_{l, \vec{\lambda}}(\cF) \otimes_{\cO_{\cB}} \cO_b \cong  \cV^\dagger_{l, \vec{\lambda}}(\cF_b), $$
where $\cF_b$ denotes the data $(\cC_b = \pi^{-1}(b); s_1(b), \ldots , s_n(b); \xi_{1|\cC_b}, \ldots , \xi_{n|\cC_b})$
consisting of a stable curve $\cC_b$ with $n$ marked points $s_1(b), \ldots, s_n(b)$ and formal coordinates
$\xi_{i|\cC_b}$ at the points $s_i(b)$.

We recall that the sheaf of conformal blocks $ \cV^\dagger_{l, \vec{\lambda}}(\cF)$  does not depend (up to a canonical isomorphism) on the formal coordinates $\xi_i$ (see e.g. \cite{U} Theorem 4.1.7). We therefore omit the formal 
coordinates in the notation.

\subsection{The projective WZW connection} \label{WZWconnection}

We now outline the definition of the projective WZW connection on the sheaf
$\cV^\dagger_{l,\vec{\lambda}}(\cF)$ over the smooth locus $\cB^s \subset \cB$
parameterizing smooth curves and refer to \cite{TUY} or \cite{U} for a detailed
account. Let $\cD \subset \cB$ be the discriminant locus and let $\cS = \coprod_{i=1}^n
s_i(\cB)$ be the union of the images of the $n$ sections. We recall the exact
sequence
\begin{equation} \label{deftheta}
0  \lra \pi_* \Theta_{\cC/ \cB}(* \cS) \lra \pi_* \Theta'_\cC(* \cS)_\pi
\stackrel{\theta}{\lra} \Theta_\cB(- \mathrm{log} \ \cD) \lra 0,
\end{equation}
where $\Theta_{\cC/\cB}(* \cS)$ denotes the sheaf of vertical rational vector fields on
$\cC$ with poles only along the divisor $\cS$, and $\Theta'_\cC(*\cS)_\pi$ the sheaf
of rational vector fields on $\cC$ with poles only along the divisor $\cS$ and with constant
horizontal components along the fibers of $\pi$. There is an $\cO_\cB$-linear map
$$ p : \pi_* \Theta'_\cC (* \cS)_\pi \lra \bigoplus_{i=1}^n \cO_\cB((\xi_i))
\frac{d}{d\xi_i},$$
which associates to a vector field $\vec{\ell}$ in $\Theta'_\cC (* \cS)_\pi$ the $n$
Laurent expansions $\ell_i \frac{d}{d\xi_i}$ around the divisor $s_i(\cB)$. Abusing notation
we also write $\vec{\ell}$ for its image under $p$
$$ \vec{\ell} = (\ell_1 \frac{d}{d\xi_1}, \cdots , \ell_n \frac{d}{d\xi_n} ) \in
\bigoplus_{i=1}^n \cO_\cB((\xi_i))  \frac{d}{d\xi_i}.$$
We then define for any vector field $\vec{\ell}$ in $\Theta'_\cC (* \cS)_\pi$ the
endomorphism $D(\vec{\ell})$ of $\cO_{\cB} \otimes \cH^\dagger_{\vec{\lambda}}$ by
$$ D(\vec{\ell}) (f\otimes u) = \theta(\vec{\ell}) . f \otimes u +
\sum_{i=1}^n f \otimes (T[\ell_i] . u) $$
for $f$ a local section of $\cO_{\cB}$ and
$u \in \cH^\dagger_{\vec{\lambda}}$. Here $T[\ell_i]$ denotes the action of the
energy-momentum tensor on the $i$-th component $\cH^\dagger_{\lambda_i}$. It is shown
in \cite{TUY} that $D(\vec{\ell})$ preserves $\cV^\dagger_{l,\vec{\lambda}}(\cF)$ and
that $D(\vec{\ell})$ only depends on the image $\theta(\vec{\ell})$ up to
homothety. One therefore obtains a projective connection $\nabla$ on the sheaf
$\cV^\dagger_{l,\vec{\lambda}}(\cF)$ over $\cB^s$ given by
$$ \nabla_{\theta(\vec{\ell})} = \theta(\vec{\ell}) + T[\vec{\ell}].$$
Since this connection is projectively flat, it induces a monodromy representation
$$ \rho_{l, \vec{\lambda}}:  \pi_1( \cB^s , b) \lra \PP \GL(\cV^\dagger_{l,\vec{\lambda}}(\cF)_b) $$
for some base point $b \in \cB^s$.

\begin{rem}
{\em For a family of smooth $n$-pointed curves of genus $0$ the projective WZW connection is actually a connection
(see e.g. \cite{U} section 5.4).}
\end{rem}

\section{Monodromy of the WZW connection for a family of $4$-pointed rational curves}

In this section we review the results by Tsuchiya-Kanie \cite{TK} on the monodromy of the
WZW connection for a family of rational curves with $4$ marked points.
We consider the
universal family $\cF_4$ over $X_3$ introduced in \eqref{univfam} with the collection
$$ \vec{\lambda}^{TK} = (\varpi, \varpi, \varpi, \varpi) \in (P_l)^4. $$
The rank of the sheaf of conformal blocks 
$\cV^\dagger_{l, \vec{\lambda}^{TK} }(\cF_4)$ equals $2$ for any  $l \geq 1$, 
 see e.g. \cite{TK} Theorem 3.3. 
Moreover, as outlined in section \ref{WZWconnection}, the bundle 
$\cV^\dagger_{l, \vec{\lambda}^{TK}}(\cF_4)$ is equipped with a flat connection $\nabla$ (not only projective).

\begin{rem}
{\em It is known \cite{TK} that the differential equations satisfied by the flat sections of
$(\cV^\dagger_{l, \vec{\lambda}^{TK}}(\cF_4),  \nabla)$ coincide with the Knizhnik-Zamolodchikov equations
(see e.g. \cite{EFK}). Moreover, we will show in a forthcoming paper that the local system 
$(\cV^\dagger_{l, \vec{\lambda}^{TK}}(\cF_4),  \nabla)$
also coincides with a certain Gauss-Manin local system.}
\end{rem}

We observe that the symmetric group $\Sigma_3$ acts naturally on the base variety $X_3$. The local system
$(\cV^\dagger_{l, \vec{\lambda}^{TK}}(\cF_4),  \nabla) $ is invariant under this $\Sigma_3$-action and admits a
natural $\Sigma_3$-linearization. Thus by descent we obtain a local system
$(\overline{\cV^\dagger_{l, \vec{\lambda}^{TK}}(\cF_4)}, \overline{\nabla})$ over  $\overline{X}_3$.
Therefore, we obtain a monodromy representation

$$\widetilde{\rho}_l : B_3 = \pi_1(\overline{X}_3, \overline{p}) \longrightarrow
\GL(\overline{\cV^\dagger_{l, \vec{\lambda}^{TK}}(\cF_4)}_{\overline{p}}) = \GL(2, \CC)$$

\begin{prop}[\cite{TK} Theorem 5.2] \label{TsuchiyaKanie}
We put $q = \exp(\frac{2i\pi}{l+2})$. There exists a basis $B$ of the vector space 
$\overline{\cV^\dagger_{l, \vec{\lambda}^{TK}}(\cF_4)}_{\overline{p}}
= \cV^\dagger_{l, \vec{\lambda}^{TK}}(\cF_4)_{p}$ such that
$$\Mat_{B}(\widetilde{\rho}_l(g_1)) =  q^{ -\frac{3}{4}} \left(
\begin{array}{cc}
q & 0  \\
0 & -1
\end{array}
 \right),
\qquad
\Mat_{B}(\widetilde{\rho}_l(g_2)) = \frac{ q^{-\frac{3}{4}}}{q+1} \left(
\begin{array}{cc}
-1 & t \\
t & q^2
\end{array}
 \right), $$
with $t = \sqrt{q(1+q+q^2)}$. Note that both matrices have eigenvalues $q^{\frac{1}{4}}$ and $-q^{-\frac{3}{4}}$.
\end{prop}

\begin{rem}\label{rem-Ka}
{\em These matrices have already been used in the paper \cite{AMU}.}
\end{rem}

\section{Infinite monodromy over $\Mfr_{0,4}$}

We denote by $\rho_l$ the restriction of the monodromy representation $\widetilde{\rho}_l$ to the subgroup
$\pi_1(\Mfr_{0,4}, x)$ of $B_3$ (see section 2.4.2)
$$ \rho_l : \pi_1(\Mfr_{0,4}, x) \subset B_3   \lra \GL(2, \CC). $$

\begin{prop} \label{sigmainf}
Let $\sigma \in \pi_1(\Mfr_{0,4}, x)$ be the element introduced in \eqref{defsigma}.
If $l   \not= 1,2,4$ and $8$, then the element $\rho_l(\sigma)$ has infinite
order in both $\PGL(2, \CC)$ and $\GL(2, \CC)$
\end{prop}

\begin{proof}
Using the explicit form of the monodromy representation $\rho_l$ given in Proposition \ref{TsuchiyaKanie}
we compute the matrix associated to $\Psi(\sigma) = \Psi(\sigma_1^{-1} \sigma_2) = g_2 g_1^{-2} g_2$
$$\Mat_{B}(\widetilde{\rho}_l(\Psi(\sigma))) = \frac{1}{(q+1)^2} \left(
\begin{array}{cc}
q^{-2} + t^2 & t(q^2 - q^{-2})  \\
t(q^2 - q^{-2}) & t^2q^{-2} + q^4
\end{array}
 \right).$$
This matrix has determinant $1$ and trace
$2 - q -q^{-1} + q^2 + q^{-2}$. Hence the matrix has finite order if and only if there exists
a primitive root of unity $\lambda$ such that
$$\lambda + \lambda^{-1} = 2 - q -q^{-1}  +q^2 + q^{-2}.$$
In \cite{Mas} it is shown that this can only happen if $l = 1,2,4$ or $8$: using the 
transitive action of $\Gal(\bar\QQ/\QQ)$ on primitive roots of unity, one gets that, if such a $\lambda$ exists for
$q = \exp(\frac{2i\pi}{l+2})$, then  for
{\em any} primitive $(l+2)$-th root $\tilde{q}$ there exists a primitive root $\tilde{\lambda}$ such that
$$ \tilde{\lambda} + \tilde{\lambda}^{-1} = 2 - \tilde{q} - \tilde{q}^{-1}  +
\tilde{q}^2 +  \tilde{q}^{-2}.$$
In particular, we have the inequality $| 1 - \RR \mathrm{e}(\tilde{q}) + \RR \mathrm{e} (\tilde{q}^2) | \leq 1$ for any
primitive $(l+2)$-th root $\tilde{q}$. But for $l \not= 1,2,4$ and $8$, one can always find a
primitive $(l+2)$-th root $\tilde{q}$ such that $\RR \mathrm{e} (\tilde{q}^2) > \RR \mathrm{e} (\tilde{q})$ --- for the
explicit root $\tilde{q}$ see \cite{Mas}.

Finally, since $\rho_l(\sigma)$ has trivial determinant, its class in $\PGL(2,\CC)$ will also have infinite order.
\end{proof}

\bigskip

\begin{rem}
{\em The same computation shows that the element $\rho_l(\sigma_1 \sigma_2^{-1}) \in  \GL(2, \CC)$ also has
infinite order if $l  \not= 1,2,4$ and $8$. This implies that the orientation chosen for both loops $\sigma_1$ and
$\sigma_2$ around $0$ and $1$ is irrelevant. On the other hand, it is immediately seen that the elements
$\rho_l(\sigma_1), \rho_l(\sigma_2)$ and $\rho_l(\sigma_1 \sigma_2)$ have finite order for
any level $l$.}
\end{rem}

\bigskip

\begin{prop} \label{finiteprojgroup}
In the four cases $l = 1,2,4$ and $8$, the image
$\rho_l(\pi_1(\Mfr_{0,4}, x))$ in the projective linear group $\PP \GL(2, \CC)$ is finite and
isomorphic to the groups given in table
$$
\begin{tabular}{|c||c|c|c|c|}
        \hline
            l & $1$ &  $2$ &  $4$ &  $8$  \\
            \hline
            $\rho_l(\pi_1(\Mfr_{0,4}, x))$ & $\mu_3$ & $\mu_2 \times \mu_2$ & $A_4$ & $A_5$  \\
     \hline
\end{tabular}
$$
Here $A_n$ denotes the alternating group on $n$ letters.
\end{prop}

\begin{proof}
We denote by $m_1, m_2 \in \PP \GL(2, \CC)$ the elements defined by the matrices
$\Mat_B(\rho_l(\sigma_1))$ and
$\Mat_B(\rho_l(\sigma_2))$ and denote by $\ord(m_i)$ their order in the group $\PP \GL(2, \CC)$.
In the first two cases one immediately checks the relations $m_1 = m_2$, $\ord(m_1) = \ord(m_2) = 3$ (for $l = 1$) and
$\ord(m_1) = \ord(m_2) = \ord(m_1m_2) = 2$ (for $l =2$).

In the case $l  = 4$ we recall that the alternating group $A_4$ has the following presentation by generators and
relations
$$ A_4 = \langle a,b \ | \ a^3 = b^2 = (ab)^3 = 1 \rangle. $$
Using the formulae of Proposition \ref{TsuchiyaKanie} and \ref{sigmainf} we check that $\ord(m_1) = \ord(m_2) = 3$ and
$\ord(m_1^{-1}m_2) = 2$, so that $a=m_1$ and $b = m_1^{-1}m_2$ generate the group $A_4$.

In the case $l = 8$ we recall that the alternating group $A_5$ has the following presentation by generators and
relations
$$ A_5 = \langle a,b \ | \ a^2 = b^3 = (ab)^5 = 1 \rangle. $$
Using the formulae of Proposition \ref{TsuchiyaKanie} and \ref{sigmainf} we check that $\ord(m_1) = \ord(m_2) = 5$ and
$\ord(m_1^{-1}m_2) = 3$. Moreover a straightforward computation shows that the element $m_1^{-1} m_2 m_1^{-1}$ is
(up to a scalar) conjugate to the matrix
$$\Mat_{B}(\widetilde{\rho}_l(g_1^{-2} g_2^2 g_1^{-2})) = * \left(
\begin{array}{cc}
q^{-4}(1 + t^2) & t(1 - q^{-2})  \\
t(1 - q^{-2}) & t^2 + q^4
\end{array}
 \right),$$
which has trace zero. Note that $t^2 = q+q^2+q^3$ and $q^{-4} = -q$. Hence $\ord(m_1^{-1} m_2 m_1^{-1} ) =
\ord(m_1 m_2^{-1} m_1) = 2$. Therefore if we put $a = m_1 m_2^{-1} m_1$ and $b = m_1^{-1} m_2$, we have
$ab = m_1$ and $ab^2 = m_2$, so that $\ord(a) = 2$, $\ord(b)= 3$, and $\ord(ab) = 5$, i.e. $a,b$ generate the
group $A_5$.
\end{proof}

\begin{cor}
In the four cases $l= 1,2,4$ and $8$, the image $\widetilde{\rho}_l(B_3)$ in $\GL(2, \CC)$ is finite.
\end{cor}

\begin{proof}
First, we observe that the image $\rho_l(\pi_1(\Mfr_{0,4}, x))$ in $\GL(2, \CC)$ is finite. In fact, by Proposition
\ref{finiteprojgroup} its image in $\PP \GL(2, \CC)$ is finite and its intersection
$\rho_l(\pi_1(\Mfr_{0,4}, x)) \cap
\CC^* \mathrm{Id}$ with the center of $\GL(2, \CC)$ is also finite. The latter follows from the fact that the determinant
$\det \Mat_{B}(\widetilde{\rho}_l(g_i)) =  - q^{- \frac{1}{2}}$ has finite order in $\CC^*$.

Secondly, we recall that $P_3$ is generated by the normal subgroup $\pi_1(\Mfr_{0,4}, x)$ and by the element $g_1^2$. Since
$\widetilde{\rho}_l(g_1^2)$ has finite order and since $B_3/P_3 = \Sigma_3$ is finite, we obtain that $\widetilde{\rho}_l(B_3)$ is a finite
subgroup.
\end{proof}

\bigskip

In the proof of the main theorem we will need the following corollary of Proposition \ref{sigmainf}. 
We consider the following compact subset of $\CC$
\begin{equation} \label{baseB}
\cB =  D \setminus (\Delta_{-1} \cup \Delta_0 \cup \Delta_1) 
\end{equation}
where $D \subset \CC$ is the closed disc centered at $0$ with radius $2$ and $\Delta_z \subset \CC$ denotes the open disc centered at $z$ with 
very small radius. We choose as base point $b = i \in \cB$.
Let $\xi \in \pi_1(\cB, b)$ be the loop going once around the two points $-1$ and $0$.  
Let $$ \cF =  (\PP^1 \times \cB \ra \cB; s_0 , s_1, s_{-1} , s_u , s_{-u}) $$
be the family of $5$-pointed rational curves, where the 5 sections $s_0 , s_1, s_{-1} , s_u , s_{-u}$ map
$u \in \cB$ to $0,1,-1, u$ and $-u$ respectively. 

\begin{cor} \label{fivepoints}
For $l \not= 1,2,4$ and $8$, the image of the 
loop $\xi \in \pi_1(\cB, b)$ under the monodromy representation 
$$ \rho_l : \pi_1(\cB, b) \lra \GL(\cV^\dagger_{l, 0,\vec{\lambda}^{TK}} (\cF)) $$
has infinite order. 
\end{cor}

\begin{proof}
Since propagation of vacua is a flat isomorphism (see e.g. \cite{Lo} Proposition 22) we can drop the 
point $0$ which is marked with the zero weight. Thus it suffices to show the statement for the same family 
with the $4$ sections $s_1, s_{-1} , s_u , s_{-u}$. The cross-ratio of the 4 points $1,-1,u,-u$ equals
$$ t = \frac{(u+1)^2}{4u} = \frac{1}{2} + \frac{1}{4}(u + u^{-1}). $$ 
We also introduce the $4$-pointed family
$$\cF' = (\PP^1 \times  \cB \ra \cB ; s_0, s_1, s_\infty , s_t), $$ 
where the section $s_t$ maps $u$ to the cross-ratio $t$ and we observe that there exists
an automorphism $\alpha : \PP^1 \times \cB \ra  \PP^1 \times \cB$ over $\cB$ (which can be made 
explicit) mapping the $4$ sections  $s_1, s_{-1}, s_u , s_{-u}$ to the $4$ sections
$s_0, s_1, s_\infty , s_t$. Moreover the automorphism $\alpha$ induces an isomorphism
between the two local systems $(\cV^\dagger_{l, \vec{\lambda}^{TK}}(\cF) , \nabla)$
and $(\cV^\dagger_{l, \vec{\lambda}^{TK}}(\cF') , \nabla)$ over $\cB$. We now consider the
map induced by the cross-ratio
$$ \Psi:  \cB \longrightarrow \PP^1 \setminus \{ 0,1, \infty \} = \Mfr_{0,4}, \qquad u \mapsto t.$$
One easily checks that the extension $\overline{\Psi}$ of $\Psi$ to $\PP^1$ gives a double cover of $\PP^1$ ramified
over $0 = \overline{\Psi}(-1)$ and $1 = \overline{\Psi}(1)$. Note that $\overline{\Psi}(0) = 
\overline{\Psi}(\infty) = \infty$. Hence $\Psi$ is an \'etale double cover over its image. The map $\Psi$
induces a map, denoted by $\Phi$, between fundamental groups
$$ \Phi: \pi_1(\cB,i) \lra \pi_1(\PP^1 \setminus \{ 0,1, \infty \} , \frac{1}{2}).$$
An elementary computations shows that $\Phi(\xi_1) = \sigma_1^2$, $\Phi(\xi_2) = \sigma_2^2$, and
$\Phi(\xi_3) = \sigma_2^{-1} \sigma_1^{-1}$, where $\xi_1$, $\xi_2$ and $\xi_3$ denote the loops
in $\cB$ going once around the points $-1$, $1$ and $0$ respectively. We recall from section 2.3 that $\sigma_1$ and
$\sigma_2$ denote the loops in $\Mfr_{0,4}$ around the points $0$ and $1$. All orientations of the loops
are the same. Hence if we denote $\xi = \xi_3 \xi_1$, the loop in $\cB$ going around the two points $-1$ and
$0$, then $\Phi(\xi) = \sigma_2^{-1} \sigma_1 = \sigma^{-1}$. Thus $\xi$ has infinite order  by Proposition
\ref{sigmainf}.
\end{proof}

\section{Infinite monodromy for higher genus}

\subsection{Deformation of families of pointed nodal curves}

In this section we explicitly describe a deformation of two families of rational nodal curves into smooth curves. These
deformations will be used in the proof of the main theorem.

\subsubsection{A family of rational curves}

We consider the family of rational curves $p : \cC \ra \AAA^1$ parameterized by the affine line $\AAA^1$
and given by the equation
$$ \cC = \mathrm{Zeros}(f) \subset \PP^2 \times \AAA^1 \qquad f = xy - \tau z^2, $$
where $(x:y:z)$ are homogeneous coordinates on the projective plane and $\tau$ is a coordinate on $\AAA^1$. We denote by
$\cC_\tau$ the fiber over $\tau \in \AAA^1$. For $\tau \not= 0$ the curve $\cC_\tau$ is a smooth conic and $\cC_0 =
L_0 \cup L_1$ is the union of two projective lines given by the equations $L_0 = \mathrm{Zeros}(y)$ and 
$L_1 = \mathrm{Zeros}(x)$. For $\tau \not= 0$ we can parameterize the smooth conic $\cC_\tau$ in the 
following way
$$ \Phi_\tau : \PP^1 \lra \cC_\tau \subset \PP^2, \qquad 
(\alpha: \beta) \mapsto (\beta^2 : \tau \alpha^2 : \alpha \beta). $$
Note that for $\tau = 0$ this morphism also gives a parametrization of the line $L_0$.

\bigskip

Let $m \geq 2$ be an integer. We define $2m+1$ sections $s_1, \ldots , s_{2m}, s_\infty$ of the family 
$p: \cC \ra \AAA^1$ parameterized by an open subset of $\AAA^1$ with coordinate $u$. We put
$$ s_1(u, \tau) = (1:\tau : 1) \qquad  s_2(u, \tau) = (1:\tau : -1), $$
$$ s_3(u, \tau) = (1:\tau u^2 : u) \qquad  s_4(u, \tau) = (1: \tau u^2 : -u), $$
and for $j = 3, \ldots , m$
$$ s_{2j-1}(u,\tau) = (\tau : j^2: j) \qquad s_{2j}(u,\tau) = (\tau  : j^2 : -j).$$
Finally we put 
$$ s_\infty(u,\tau) = (0:1:0).$$
We observe that for $\tau \not= 0$ the $2m+1$ points $s_1(u,\tau), \ldots , s_{2m}(u, \tau), s_\infty(u, \tau)$
correspond to the following points in $\PP^1$ via the morphism $\Phi_\tau$:
$$ 1,-1, u, -u, 3\tau^{-1}, -3 \tau^{-1}, \ldots, m \tau^{-1}, - m \tau^{-1}, \infty. $$
For $\tau = 0$ the points $s_1(u,0), s_2(u,0), s_3(u,0), s_4(u,0) \in L_0$ have coordinates $1,-1,u,-u$
and the points $s_5(u,0), \ldots , s_{2m}(u,0), s_\infty(u,0) \in L_1$ have coordinates $3, - 3, \ldots ,
m, - m, \infty$, in particular they do not depend on $u$. We consider the open
subset $\Omega = \{ \tau \ : \  |\tau| < \frac{1}{2} \} \subset \AAA^1$ and the compact subset 
$\cB$ as defined in \eqref{baseB},  and define the 
family 
$$ \cF_{2m+1}^{rat} = (\pi = id \times p : \cB \times \cC_{|\Omega} \ra \cB \times \Omega; s_1 , s_2, \ldots ,
s_{2m}, s_\infty) $$
of $2m+1$-pointed rational curves.

\subsubsection{A family of hyperelliptic curves}

Let $g \geq 2$ be an integer and let $\alpha$ be a complex number satisfying $|\alpha| > 1$.
We consider the family $p: \cC \ra \AAA^1 \times \AAA^1$ of curves parameterized by two complex numbers 
$(u, \tau) \in \AAA^1 \times \AAA^1$ and such that the fiber $\cC_{(u, \tau)}$ is the double cover of $\PP^1$
ramified over the $2g+2$ points:
$$ 0, \infty, u^2 + \tau , u^2 - \tau, 1 + \tau, 1 - \tau, (3\alpha)^2 + \tau, (3\alpha)^2 - \tau, \ldots ,
(g \alpha)^2 + \tau, (g \alpha)^2 - \tau.  $$
We assume that these points are distinct.
We denote the projection $pr: \cC_{(u, \tau)} \ra \PP^1$.
The family of curves $\cC$ can be constructed by taking the closure in $\PP^2 \times \AAA^1 \times \AAA^1$ of the affine
curve in $\AAA^2 \times \AAA^1 \times \AAA^1$ over $\AAA^1 \times \AAA^1$ defined by the equation
\begin{equation} \label{hyperelliptic}
y^2 =  x (x -1 + \tau) (x -1 - \tau) (x - u^2 + \tau )(x - u^2 - \tau) \prod_{j= 3}^g 
(x - (j\alpha)^2 + \tau )(x - (j \alpha)^2 - \tau) 
\end{equation}
and by blowing up $g$ times the singular point at $\infty$. 

\bigskip

We notice that for any $\tau \not= 0$ with $|\tau|$ sufficiently small and for $u$ varying in a
Zariski open subset $U_\tau$ of $\AAA^1$ the curves 
$\cC_{(u, \tau)}$ are smooth hyperelliptic curves of genus $g$. For $\tau = 0$ and 
$u^2 \not=  0,1, (3\alpha)^2 , \ldots , (g \alpha)^2$ 
the curve $\cC_{(u, \tau)}$ 
is a rational nodal curve with $g$ nodes lying over the points 
$1, u^2, (3\alpha)^2 , \ldots , (g \alpha)^2$ of $\PP^1$.
The normalization map $\eta: \PP^1 \ra \cC_{(u, 0)}$ is explicitly given over $\AAA^1 \subset \PP^1$ by
the expressions
$$ \eta: \AAA^1 \ra \AAA^2, \qquad t \mapsto (x(t),y(t)) = ( t^2, t(t^2 -1)(t^2 - u^2) \prod_{j=3}^g 
(t^2 - (j \alpha)^2 ). $$
This shows that the pre-images by $\eta$ of the $g$ nodes are 
$$1,-1, u,-u, 3\alpha, - 3\alpha, \ldots , g \alpha, - g \alpha. $$

\bigskip

We put $s_\infty(u,\tau) = \infty \in \cC_{(u, \tau)}$ for any $(u, \tau)$ and $\cB$ as defined in \eqref{baseB}. Let $\Omega = \{ \tau : |\tau| < \frac{1}{2} \}$.  We define the family 
$$ \cF_g^{hyp} := ( \pi : \cC_{| \cB \times \Omega} \ra \cB \times \Omega , s_\infty) $$
of $1$-pointed hyperelliptic curves.

\bigskip

\subsection{The sewing procedure}

We will briefly sketch the construction of the sewing map and give some of its
properties (for the details see \cite{TUY} or \cite{U}).

\bigskip

We consider a flat family
$$ \cF = ( \pi: \cC \ra \cB \times \Omega; s_1, \ldots , s_n) $$
of $n$-pointed connected projective curves parameterized by  $\cB \times \Omega$, where $\cB$
is a complex manifold and $\Omega \subset \AAA^1$ is an open subset of the complex affine line $\AAA^1$
containing the origin $0$. We assume that the family $\cF$ satisfies the following conditions:

\begin{enumerate}
\item the curve $\cC_{(b,\tau)}$ is smooth if $\tau \not= 0$.
\item the curve $\cC_{(b,0)}$ has exactly one node.
\end{enumerate}

We also introduce the family $\widetilde{\cF}$ of $n+2$-pointed curves associated to $\cF$
$$ \widetilde{\cF} = ( \widetilde{\pi}: \widetilde{\cC} \ra \cB ; s_1, \ldots , s_{n+2}) $$
which desingularizes the family of nodal curves $\cF_{|\cB \times \{ 0 \}}$. Here $s_{n+1}(b)$ and
$s_{n+2}(b)$ are the two points of $\widetilde{\cC}_b$ lying over the node of $\cC_{(b,0)}$.

\begin{rem}
{\em An example of a family $\cF$  satisfying the above conditions is given in
section 6.1.1.}
\end{rem}

For any dominant weight $\mu$ the
Virasoro operator $L_0$ induces a
decomposition of the representation space $\cH_\mu$ into a direct sum
of eigenspaces $\cH_\mu(d)$ for the eigenvalue $d + \Delta_\mu$ of $L_0$, where
$\Delta_\mu \in \QQ$ is the
trace anomaly and $d \in \NN$. We recall that there exists a unique (up to a
scalar) bilinear pairing 
$$(. | .) : \cH_\mu \times \cH_{\mu^\dagger} \ra \CC \qquad \text{such that} \qquad
(X(n) u | v) + (u|X(-n)v) = 0$$ 
for any $X \in \slfr(2)$, $n \in \ZZ$, $u \in \cH_\mu$,
$v \in \cH_{\mu^\dagger}$ and $(. |. )$ is zero on
$\cH_\mu(d) \times \cH_{\mu^\dagger}(d')$ if $d \not= d'$. We choose a basis $\{ v_1(d), \ldots,
v_{m_d}(d) \}$ of $\cH_\mu(d)$ and let $\{ v^1(d), \ldots,
v^{m_d}(d) \}$ be its dual basis of $\cH_{\mu^\dagger}(d)$ with respect to the above
bilinear form. Then the element
$$ \gamma_d = \sum_{i=1}^{m_d} v_i(d) \otimes v^i(d) \in \cH_\mu (d) \otimes
\cH_{\mu^\dagger} (d) \subset \cH_\mu \otimes
\cH_{\mu^\dagger}$$
does not depend on the basis. We recall that 
$\cV^\dagger_{l,\vec{\lambda}, \mu, \mu^\dagger}(\widetilde{\cF})$ is a locally free
$\cO_{\cB}$-module. Given a section $\psi \in \cV^\dagger_{l,\vec{\lambda}, \mu, \mu^\dagger}(\widetilde{\cF})$ 
we define an $\cH^\dagger_{\vec{\lambda}}$-valued power series
$\widetilde{\psi} \in 
\cH^\dagger_{\vec{\lambda}}[[\tau]] \otimes \cO_{\cB}$ as follows. For any non-negative integer $d$
the inclusion $\cH_{\vec{\lambda}} \hookrightarrow \cH_{\vec{\lambda}, \mu, \mu^\dagger}$,
$v \mapsto v \otimes \gamma(d)$ induces a dual projection 
\begin{equation} \label{projectionpid}
\pi_d : 
\cH^\dagger_{\vec{\lambda}, \mu, \mu^\dagger} \lra \cH^\dagger_{\vec{\lambda}}.
\end{equation}
We denote by $\psi_d$ the image $\pi_d(\psi) \in \cH^\dagger_{\vec{\lambda}} \otimes \cO_{\cB}$. 
We then define
$$ \widetilde{\psi}  =  \sum_{d=0}^\infty  \psi_d \tau^d \in \cH_{\vec{\lambda}}[[\tau]] \otimes
\cO_{\cB}.$$
It is shown in \cite{TUY} that 
$$\widetilde{\psi} \in \cV^\dagger_{l, \vec{\lambda}} (\cF) \otimes_{\cO_{\cB \times \Omega}} \widehat{\cO},$$
where $\widehat{\cO}$ denotes the structure sheaf of the completion of $\cB \times \Omega$ along the
divisor $\cB \times \{ 0 \}$. Note that $\widehat{\cO} = \cO_{\cB}[[\tau]]$. Therefore 
we obtain for any $\mu \in P_l$ and any $\vec{\lambda} \in (P_l)^n$ an $\cO_{\cB \times \Omega}$-linear map
$$ s_\mu : \cV^\dagger_{l,\vec{\lambda},\mu,  \mu^\dagger}(\widetilde{\cF}) \otimes_{\cO_{\cB}} 
\cO_{\cB \times \Omega} \lra
\cV^\dagger_{l, \vec{\lambda}} (\cF) \otimes_{\cO_{\cB \times \Omega}} \widehat{\cO} , \qquad \psi \mapsto s_\mu(\psi) =  
\widetilde{\psi}, $$
called the {\em sewing map}. We denote $\Omega^0 = \Omega \setminus \{ 0 \}$.

\bigskip

We recall that the sheaf $\cV^\dagger_{l,\vec{\lambda},\mu,  \mu^\dagger}(\widetilde{\cF})$
over $\cB$, as well as its pull-back to the product $\cB \times \Omega^0$ under the first
projection, is equipped with the WZW-connection (see section 3.2). On the other hand,
the restriction of the sheaf $\cV^\dagger_{l, \vec{\lambda}} (\cF)$ to $\cB \times \Omega^0$,
which is the open subset of $\cB \times \Omega$ parameterizing smooth curves, is also equipped
with the WZW-connection. The main result of this section (Theorem \ref{sewingflat}) says that the sewing
map $s_\mu$ is projectively flat for both connections. We first need to recall the following

\begin{thm}[\cite{TUY} Theorem 6.2.2] \label{sewingisflat}
For any section $\psi \in \cV^\dagger_{l,\vec{\lambda}, \mu, \mu^\dagger}(\widetilde{\cF})$
the multi-valued formal power series $\widehat{\psi} = \tau^{\Delta_\mu} \widetilde{\psi}$
has the following properties :
\begin{enumerate}
\item it satisfies the relation
$$ \nabla_{\tau \frac{d}{d\tau}} (\widehat{\psi}) = 0 \qquad \mathrm{mod} \
\cO_{\cB \times \Omega} \widehat{\psi}.$$
\item for any $b \in \cB$, the power series $\widetilde{\psi}_b$ converges. 
\item if $\cB$ is compact, there
exists a non-zero positive real number $r$ such that the power series $\widetilde{\psi}$ defines a
holomorphic section of $\cV^\dagger_{l, \vec{\lambda}} (\cF)$ over $\cB \times D_r$,
where $D_r \subset \Omega$ is the open disc centered at $0$ with radius $r$.
\end{enumerate}
\end{thm}

\begin{proof}
Only part (3) is not proved in \cite{TUY} Theorem 6.2.2. Consider a point $a \in \cB$. We choose holomorphic
coordinates $u_1, \ldots, u_m$ centered at the point $a \in \cB$. Locally around the point $a \in \cB$ the
section $\widetilde{\psi}$ can be expanded as a $\cH^\dagger_{\vec{\lambda}}$-valued power series in the $m+1$
variables $u_1, u_2, \ldots, u_m, \tau$. Given a second point $b \not= a$ with coordinates $b= (b_1, \ldots, b_m)$ with 
$b_i \not= 0$, we know by part (2) that $\widetilde{\psi}_b$ converges if $|\tau| \leq \rho $ for some real $\rho$. By the general theory
of functions in several complex variables (see e.g. \cite{O} Proposition 1.2) we deduce that  $\widetilde{\psi}_c$
converges for $|\tau| < \rho$ and for any $c = (c_1, \ldots , c_m)$ such that $|c_i| < |b_i|$. Therefore, there
exists for any $a \in \cB$ a polydisc $\Delta_a$ around $a$ and a real number $r_a$ such that the radius
of convergence of the series $\widetilde{\psi}_c$ for any $c \in \Delta_a$ is at least $r_a$. By considering the
covering of $\cB$ by the polydiscs $\Delta_a$ and by the fact that $\cB$ is compact, we then obtain the desired
non-zero real number $r$.
\end{proof}

\begin{rem}
{\em We note that the statement given in \cite{TUY} Theorem 6.2.2 says that there exists
a vector field $\vec{\ell}$ over the family of curves $\mathcal{C}$ such that
$$ \left(-\tau \frac{d}{d\tau} + T[\vec{\ell}]  \right) . \widehat{\psi} = 0 \qquad
\mathrm{mod} \ \cO_{\cB \times \Omega} \widehat{\psi},$$
which is equivalent to the above statement using the property $\theta(\vec{\ell}) =
-\tau \frac{d}{d\tau}$. This last equality is actually proved in \cite{TUY} Corollary 6.1.4, but
there is a sign error. The correct formula of \cite{TUY} Corollary 6.1.4 is
$\theta(\vec{\ell}) =  -\tau \frac{d}{d\tau}$, which is obtained by writing the
$1$-cocycle $\theta_{12}(u,\tau) = \tilde{\ell}'_{u,\tau |U_2} - \tilde{\ell}_{u,\tau |U_1}$.}
\end{rem}

\begin{rem}
{\em By making the base change $\nu^j = \tau$, where $j$ is the denominator of the trace
anomaly $\Delta_\mu$, we obtain a section
$\widehat{\psi} \in \cV^\dagger_{l, \vec{\lambda}} (\cF) \otimes_{\cO_{\cB \times \Omega}} \cO_\cB[[\nu]]$ 
satisfying $\nabla_{\nu \frac{d}{d\nu}} (\widehat{\psi}) = 0 \ \mathrm{mod} \
\cO_{\cB \times \Omega}  \widehat{\psi}$.}
\end{rem}

The next result says that the sewing map is projectively flat.

\begin{thm} \label{sewingflat}
For any $\mu \in P_l$ and any $\vec{\lambda} \in (P_l)^n$ the restriction of the sewing map $s_\mu$ to the open subset 
$\cB \times \Omega^0$
$$ s_\mu : \cV^\dagger_{l,\vec{\lambda},\mu,  \mu^\dagger}(\widetilde{\cF}) \otimes_{\cO_{\cB}} 
\cO_{\cB \times \Omega^0} \lra
\cV^\dagger_{l, \vec{\lambda}} (\cF) \otimes_{\cO_{\cB \times \Omega^0}} \widehat{\cO}
$$
is projectively flat for the WZW connections on both sheaves of conformal blocks.
\end{thm}

\begin{proof}
We need to check that $\nabla_{D}(\widetilde{\psi}) = 0 \  \mathrm{mod} \
\cO_{\cB \times \Omega^0} \widetilde{\psi}$ if $\nabla_D(\psi) = 0 \  \mathrm{mod} \
\cO_{\cB \times \Omega^0} \psi$ for any vector field $D$ over $\cB \times \Omega^0$. By 
$\cO_{\cB \times \Omega^0}$-linearity of the connection, it suffices to check the following two points:
\begin{enumerate}
\item $\nabla_{\frac{d}{d\tau}}(\widetilde{\psi}) = 0  \ \mathrm{mod} \ \cO_{\cB \times \Omega^0} \widetilde{\psi}$ for any
section $\psi$.
\item $\nabla_{\partial}(\widetilde{\psi}) = 0 \  \mathrm{mod} \  \cO_{\cB \times \Omega^0} \widetilde{\psi}$ if 
$\nabla_{\partial}(\psi) = 0 \  \mathrm{mod} \ \cO_{\cB} \psi$ for any vector field $\partial$ on $\cB$.
\end{enumerate}

Part (1) is an immediate corollary of the previous Theorem \ref{sewingisflat} (1). Note that
$\nabla_{\frac{d}{d\tau}}(\psi) = 0$ for any section $\psi$ over $\cB$.

We now prove part (2). We start with a lemma, which is an analogue of \cite{U} Lemma 5.3.1. 

\begin{lem} \label{vectorfield}
Let $b \in \cB$ and let $\partial$ be a vector field in some neighbourhood $U$ of $b$. If we choose $U \subset \cB$
sufficiently small, then there exist local coordinates $(u_1, \ldots, u_m, z)$ (resp. $(u_1, \ldots, u_m, w)$) of a neighbourhood
$X$ (resp. $Y$) of $s_{n+1}(U) \subset \widetilde{\cC}_{|U}$ (resp.  $s_{n+2}(U) \subset \widetilde{\cC}_{|U}$) and a vector
field $\vec{\ell}$ over $\widetilde{\cC}_{|U}$, which is constant along the fibers
$$ \vec{\ell} \in H^0(\widetilde{\cC}_{|U}, 
\Theta'_{\widetilde{\cC}}(* \sum_{i=1}^n s_i(\cB))_{\widetilde{\pi}})$$
and which satisfy the following conditions :
\begin{enumerate}
\item the sections $s_{n+1}$ and $s_{n+2}$ are given by the mappings
$$ s_{n+1} : (u_1, \ldots , u_m) \mapsto (u_1, \ldots , u_m,0)= (u_1, \ldots , u_m,z) $$
$$ s_{n+2} : (u_1, \ldots , u_m) \mapsto (u_1, \ldots , u_m,0)= (u_1, \ldots , u_m,w) $$
\item $\vec{\ell}_{|X} = z \frac{d}{dz} + \partial$, $\vec{\ell}_{|Y} = -w \frac{d}{dw} + \partial$. In particular,  $\theta(\vec{\ell}) = \partial$, i.e.
$\vec{\ell}$ projects onto the vector field $\partial$. Here $\theta$ denotes the projection on the horizontal component, see
\eqref{deftheta}.
\end{enumerate}
\end{lem}

\begin{proof}
The proof follows the lines of the proof of \cite{U} Lemma 5.3.1. 

For a small neighbourhood $U$ of $b$, we choose $(u_1, \ldots, u_m, x)$ and $(u_1, \ldots , u_m, y)$  local coordinates in
$\widetilde{\pi}^{-1}(U)$ satisfying condition (1) of the Lemma. We denote by $\widetilde{\pi}' : \widetilde{\cC}' \lra \cB$ the
family of nodal curves parameterized by $\cB$ obtained from the family $\widetilde{\cC}$ by identifying the two divisors
$s_{n+1}(\cB)$ and $s_{n+2}(\cB)$. Note that we have a sequence of maps over $\cB$
$$ \widetilde{\cC} \map{\nu} \widetilde{\cC}' \lra \cC_{|\cB \times \{ 0 \}}.$$
By choosing $U$ small enough, we can lift the vector field $\partial$ over $U \subset \cB$ to a vector field $\vec{\ell}$
over $\widetilde{\cC}'_{|U}$ which is constant along the fibers of $\widetilde{\cC}'_{|U} \ra U$ and has poles only at 
$\cS_{|U}$, i.e. lies in $\Theta'_{\widetilde{\cC}'}(*\cS)_{\widetilde{\pi}'}$. The inclusion
$$ \Theta'_{\widetilde{\cC}'}(*\cS)_{\widetilde{\pi}'} \hookrightarrow \nu_*   
\Theta'_{\widetilde{\cC}} (*\cS - s_{n+1}(\cB) - s_{n+2}(\cB))_{\widetilde{\pi}} $$
allows us to see $\vec{\ell}$ as a vector field over $\widetilde{\cC}_{|U}$ having the property
$$\frac{da}{dx} (u,0) + \frac{db}{dy} (u,0) = 0, $$
where the functions $a(u,x)$ and $b(u,y)$ are defined by the expressions of the restriction of
$\vec{\ell}$ to the neighboorhoods $X$ and $Y$
$$ \vec{\ell}_{|X} = a(u,x) \frac{d}{dx} \ \text{and} \  \vec{\ell}_{|Y} = b(u,y) \frac{d}{dy}. $$
Note that $a(u,0) = 0$ and $b(u,0) = 0$ for any $u \in U$. The rest of the proof then goes as in 
\cite{U} Lemma 5.3.1 or \cite{TUY} Lemma 6.1.2.
 \end{proof}

Let $b \in \cB$ and let $\partial$ be a vector field in some neighbourhood $U$ of $b$. Taking $U$ sufficiently small, we can
lift the projective connection $\nabla$ on the sheaf 
$\cV^\dagger_{l,\vec{\lambda},\mu,  \mu^\dagger}(\widetilde{\cF})$ over $U$ to a connection. We consider the vector field
$\vec{\ell}$ constructed in Lemma \ref{vectorfield}. Then for a local section $\psi$ over $U$ the equation
$$\nabla_{\partial}(\psi) = 0 \  \mathrm{mod} \ \cO_{U} \psi $$
is equivalent to the equation
\begin{equation} \label{flatsection}
 \partial \psi + T[\vec{\ell}] \psi + a \psi = 0 
\end{equation}
for some local section $a$ of $\cO_U$. We will take as local coordinates $\xi_{n+1} = z$ and $\xi_{n+2} = w$ 
around the divisors $s_{n+1}(U)$ and $s_{n+2}(U)$, as introduced
in Lemma \ref{vectorfield}. Then the image of $\vec{\ell}$ under $p$ can be written 
$$ \vec{\ell} = (\ell_1 \frac{d}{d\xi_1}, \cdots , \ell_n \frac{d}{d\xi_n}, \xi_{n+1} \frac{d}{d\xi_{n+1}} ,
 - \xi_{n+2} \frac{d}{d\xi_{n+2}} )  $$
Since $T[ \xi_{i} \frac{d}{d\xi_{i}}] = L_0$ acting on the $i$-th component of the tensor product, we obtain 
the following decomposition
$$T[\vec{\ell}] = \sum_{i=1}^n T[\ell_i] + L_0^{(n+1)} - L_0^{(n+2)},$$
where the exponent $(i)$ of the Virasoro operator $L_0$ denotes an action on the $i$-th component. 

For any non-negative integer $d$ we then project equation \eqref{flatsection} via the map $\pi_d$ defined
in \eqref{projectionpid} into $\cH^\dagger_{\vec{\lambda}}$, which leads to
$$ \partial \psi_d + \sum_{i=1}^n T[\ell_i] \psi_d + \pi_d ( L_0^{(n+1)} \psi )  -
\pi_d ( L_0^{(n+2)} \psi ) + a \psi_d = 0. $$
We have the equalities $\pi_d ( L_0^{(n+1)} \psi ) = (\Delta_\mu + d) \psi_d$ and
$\pi_d ( L_0^{(n+2)} \psi ) = (\Delta_{\mu^\dagger} + d) \psi_d$. Hence both terms cancel,
since $\Delta_\mu = \Delta_{\mu^\dagger}$. This leads to the equations for any $d$
\begin{equation} \label{flatsection2}
\partial \psi_d + \sum_{i=1}^n T[\ell_i] \psi_d  + a \psi_d = 0. 
\end{equation}
Multiplying \eqref{flatsection2} with $\tau^d$ and summing over $d$, we obtain the equation
\begin{equation} \label{flatsectiontilde}
\partial \widetilde{\psi} + \sum_{i=1}^n T[\ell_i] \widetilde{\psi}  + a \widetilde{\psi} = 0.
\end{equation}
Note that $\partial (\psi_d \tau^d) = (\partial \psi_d) \tau^d$, since the vector 
field $\partial$ comes from $\cB$.

The vector field $\vec{\ell}$ over $\widetilde{\cC}_{|U}$ determines a vector field $\vec{m}$ over the family 
of smooth curves $\cC_{|U \times \Omega^0}$ as follows. We fix a point $b \in \cB$ and a non-zero
complex number $\tau$ with $|\tau| < 1$. The smooth curve $\cC_{(b, \tau)}$ is obtained from
the curve $\widetilde{\cC}_b$ by removing the two closed discs $D_{n+1}$ and $D_{n+2}$ centered 
at $s_{n+1}(b)$ and $s_{n+2}(b)$ with radius $|\tau|$, and by identifying in the open curve
$\widetilde{\cC}_b \setminus (D_{n+1} \cup D_{n+2})$ the two annuli
$$ A_{n+1} = \{ p \in \widetilde{\cC}_b \ : \ |\tau| < |z(p)| < 1 \} \ \text{and} \
 A_{n+2} = \{ p \in \widetilde{\cC}_b \ : \ |\tau| < |w(p)| < 1 \} $$
according to the relation
$$ zw = \tau.$$
Under this identification, we see that the two restrictions of vector fields 
$\vec{\ell}_{|\{ b \} \times A_{n+1}}$ and  $\vec{\ell}_{|\{ b \} \times A_{n+2}}$ 
correspond (since $z \frac{d}{dz} = - w \frac{d}{dw}$) and thus define a vector field
$\vec{m}$ over $\cC_{(b, \tau)}$, which has poles only at the $n$ points $s_1(b), \ldots ,
s_n(b)$. Moreover the Laurent expansion of $\vec{m}$ at $s_1(b), \ldots ,
s_n(b)$ coincide with the Laurent expansion of $\vec{\ell}$. For the construction in a 
family, see \cite{U} section 5.3. Hence $\theta(\vec{m}) = \partial$ and $p(\vec{m}) = 
(\ell_1 \frac{d}{d\xi_1}, \cdots , \ell_n \frac{d}{d\xi_n})$. So equation \eqref{flatsectiontilde}
can be written as
$$ \nabla_\partial \widetilde{\psi} = \partial \widetilde{\psi}  + T[\vec{m}] \widetilde{\psi} = 0  
 \  \mathrm{mod} \  \cO_{\cB \times \Omega^0} \widetilde{\psi}.$$
The last equation means that $\widetilde{\psi}$ is a projectively flat section for the WZW
connection.
\end{proof}

From now on we assume that $\cB$ is compact.
Since by Theorem \ref{sewingisflat} (3) the formal power series $\widetilde{\psi}$ determines a holomorphic section
over $\cB \times D_r$ we can choose a complex number $\tau_0 \not= 0$ with $|\tau_0| < r$ and evaluate $\widetilde{\psi}$ at 
$\tau_0$. This gives a section $\widetilde{\psi}(\tau_0)$ of the conformal block 
$\cV^\dagger_{l, \vec{\lambda}} (\cF_{\tau_0}) = \cV^\dagger_{l, \vec{\lambda}} (\cF)_{| \cB \times \{ \tau_0 \}}$.

\bigskip

Moreover, using the factorization rules (see e.g. \cite{TUY} Theorem 6.2.6 or
 \cite{U} Theorem 4.4.9) we obtain
by summing over all dominant weights $\mu \in P_l$ an $\cO_{\cB}$-linear isomorphism 
$$ \oplus s_\mu (\tau_0) :
\bigoplus_{\mu \in P_l} \cV^\dagger_{l,\vec{\lambda},\mu,  \mu^\dagger}(\widetilde{\cF})  \stackrel{\sim}{\longrightarrow}  \cV^\dagger_{l, \vec{\lambda}} (\cF_{\tau_0}), $$
which is projectively flat for the WZW connections on both sheaves over $\cB$ by Theorem \ref{sewingflat}.
We fix a base point $b \in \cB$, which gives
a direct sum decomposition 

\begin{equation} \label{factrules}
\bigoplus_{\mu \in P_l} \cV^\dagger_{l,\vec{\lambda},\mu ,  \mu^\dagger}(\widetilde{\cF})_b
\stackrel{\sim}{\longrightarrow}  \cV^\dagger_{l, \vec{\lambda}}
(\cF_{\tau_0})_b.
\end{equation}

We denote by $\DD$ the subgroup of $\PGL(\cV^\dagger_{l, \vec{\lambda}}(\cF_{\tau_0})_b)$ consisting of projective
linear maps preserving the direct sum decomposition \eqref{factrules} and by $p_\mu : \DD \lra
 \PGL( \cV^\dagger_{l, \vec{\lambda}, \mu, \mu^\dagger}(\widetilde{\cF})_b)$ the
 projection onto the summand corresponding to $\mu \in P_l$.

\bigskip

The next proposition is an immediate consequence of the fact that the maps $s_\mu(\tau_0)$ are projectively flat.

\begin{prop} \label{imagemono}
With the above notation we have  for
any $\mu \in P_l$ and any $\vec{\lambda} \in (P_l)^n$
\begin{enumerate}
\item the monodromy representation of the sheaf of conformal blocks $\cV^\dagger_{l, \vec{\lambda}}(\cF_{\tau_0})$ over
$\cB \times \{ \tau_0 \}$ takes values in the subgroup $\DD$, i.e.,
$$ \rho_{l, \vec{\lambda}} : \pi_1(\cB,b) \lra \DD \subset \PGL ( \cV^\dagger_{l, \vec{\lambda}}(\cF_{\tau_0})_b). $$
\item we have a commutative diagram
$$
\begin{CD}
\pi_1(\cB,b) @>\rho_{l, \vec{\lambda}}>> \DD \\
@VV = V @VVp_{\mu}V \\
\pi_1(\cB,b) @>\rho_{l, \vec{\lambda}, \mu, \mu^\dagger}>> \PGL( \cV^\dagger_{l, \vec{\lambda}, \mu, \mu^\dagger}(\widetilde{\cF})_b)
\end{CD}
$$
\end{enumerate}
\end{prop}

\bigskip

In the proof of the main theorem we will use the above proposition for a slightly more general family
$\cF$ of $n$-pointed connected projective curves. We shall assume that $\cF$ satisfies the two conditions:

\begin{enumerate}
\item the curve $\cC_{(b,\tau)}$ is smooth if $\tau \not= 0$.
\item the curve $\cC_{(b,0)}$ has exactly $m$ nodes.
\end{enumerate}

The desingularizing family $\widetilde{\cF}$ will thus be a $n+2m$-pointed family.

\begin{rem}
{\em An example of  a family $\cF$ satisfying the above conditions is given in
section 6.1.2.}
\end{rem}

Given an $m$-tuple $\vec{\mu} = (\mu_1, \ldots , \mu_m) \in (P_l)^m$ of dominant weights, we denote
$\vec{\mu}^\dagger = (\mu_1^\dagger, \ldots , \mu_m^\dagger) \in (P_l)^m$. The next proposition is
shown along the same lines as Proposition \ref{imagemono}. First we note that there is a decomposition

\begin{equation} \label{factrules2}
\bigoplus_{\mu \in (P_l)^m} \cV^\dagger_{l,\vec{\lambda}, \vec{\mu},  \vec{\mu}^\dagger}(\widetilde{\cF})_b
\stackrel{\sim}{\longrightarrow}  \cV^\dagger_{l, \vec{\lambda}} (\cF_{\tau_0})_b.
\end{equation}
We denote by $\DD$ the subgroup of $\PGL(\cV^\dagger_{l, \vec{\lambda}}(\cF_{\tau_0})_b)$ preserving the above 
decomposition.

\begin{prop} \label{imagemono2}
With the above notation we have  for
any $\vec{\mu} \in (P_l)^m$ and any $\vec{\lambda} \in (P_l)^n$
\begin{enumerate}
\item the monodromy representation of the sheaf of conformal blocks $\cV^\dagger_{l, \vec{\lambda}}(\cF_{\tau_0})$ over
$\cB \times \{ \tau_0 \}$ takes values in the subgroup $\DD$, i.e.,
$$ \rho_{l, \vec{\lambda}} : \pi_1(\cB,b) \lra \DD \subset \PGL ( \cV^\dagger_{l, \vec{\lambda}}(\cF_{\tau_0})_b). $$
\item we have a commutative diagram
$$
\begin{CD}
\pi_1(\cB,b) @>\rho_{l, \vec{\lambda}}>> \DD \\
@VV = V @VVp_{\mu}V \\
\pi_1(\cB,b) @>\rho_{l, \vec{\lambda}, \vec{\mu}, \vec{\mu}^\dagger}>> \PGL( \cV^\dagger_{l, \vec{\lambda}, \vec{\mu}, 
\vec{\mu}^\dagger}(\widetilde{\cF})_b)
\end{CD}
$$
\end{enumerate}
\end{prop}

\subsection{Proof of the Theorem}

We will now prove the theorem stated in the introduction. We know by \cite{L}
assuming\footnote{In fact in \cite{L} one makes the assumption $g\geq 3$ for simplicity. It can be shown that
the isomorphism also holds for $g=2$.} $g \geq 2$ that there is a projectively flat isomorphism between the two projectivized
vector bundles
$$ \PP \cZ_l \stackrel{\sim}{\longrightarrow} \PP \cV^\dagger_{l, \emptyset} $$
equipped with the Hitchin connection and the WZW connection respectively. Here
$\cV^\dagger_{l, \emptyset}$ stands for the sheaf of conformal blocks
$\cV^\dagger_{l, 0}(\cF)$ associated
to the family $\cF = ( \pi: \cC \ra \cB ; s_1)$ of curves with one point labeled
with the trivial representation $\lambda_1 = 0$ (propagation of vacua). 

\bigskip

We consider the family of hyperelliptic curves $\cF = \cF_g^{hyp}$
defined in section 6.1.2 for some $\alpha$ with $|\alpha| > 1$. The family $\widetilde{\cF}$ which
desingularizes the nodal curves $\cF_{|\cB \times \{ 0 \} }$ is a family of $(2g+1)$-pointed rational
curves with points 
$$ \infty, 1,-1, u,-u, 3\alpha, - 3\alpha, \ldots , g \alpha, - g \alpha. $$
We then deduce from  Proposition \ref{imagemono2} (2) applied to the family $\cF = \cF^{hyp}_g$ with $n= 1$, $m= g$ and the choice of
weights $\lambda_1 = 0$ and $\mu_1 = \cdots = \mu_g = \varpi$, where we associate the weight $0$ to $\infty$ and
the weight $\varpi$ to  the remaining $2g$ points (note that $\varpi = \varpi^\dagger$), that
it suffices to show that the monodromy representation 
$$ \pi_1(\cB,b) \lra  \PGL( \cV^\dagger_{l, 0 , \varpi , \ldots , \varpi }(\widetilde{\cF})_b) $$
has an element of infinite order in its image. 

\bigskip

In order to show the last statement we consider the family of rational curves $\cF^{rat}_{2g+1}$ defined in 
section 6.1.1. The family $\widetilde{\cF}$ which
desingularizes the nodal curves $\cF_{|\cB \times \{ 0 \} }$ is a family of $(2g+3)$-pointed rational
curves consisting of the disjoint union of two projectve lines with $5$ points $0, 1,-1, u,-u$ on one projective line and $2g -2$ points  $0, \infty , 3, - 3, \ldots , g, -g$ on the second projective line.

\bigskip

Next we observe that the conformal block for the projective line  with $2g-2$ marked  points $0, \infty, 3, - 3, \ldots ,
g, - g$ with the zero weight at the points $0$ and $\infty$ and the weight $\varpi$ at the other
$2g-4$ points is non-zero. This follows from an iterated use of the propagation of vacua, 
the factorization rules and from the fact that $\dim \cV_{l,\varpi, \varpi}(\PP^1) = 1$. 

\bigskip

The previous observation together with Corollary \ref{fivepoints} then implies that the
family $\widetilde{\cF}$ of $(2g+2)$-pointed curves with weights 
$0, \varpi , \varpi , \varpi, \varpi$
on the $5$ points $0,1,-1,u,-u$ of the first projective line and weights
$0,0, \varpi, \ldots , \varpi$ on the $2g-2$ points $0, \infty, 3, - 3, \ldots ,
g, - g$ on the $2g-2$ points on the second
projective line has infinite monodromy.
We then deduce from Proposition \ref{imagemono} (2) applied to the family $\cF = \cF^{rat}_{2g+1}$ with 
$\vec{\lambda} = (\varpi, \ldots, \varpi)$  and $\mu = 0$ that the monodromy representation 
$$ \pi_1(\cB,b) \lra  \PGL( \cV^\dagger_{l, 0,  \varpi, \ldots , \varpi}(\widetilde{\cF})_b) $$
has an element of infinite order in its image, which completes the proof.

\begin{rem} \label{explicitfamily}

{\em For the convenience of the reader we recall that we have taken the family of smooth hyperelliptic curves
given by the affine equation \eqref{hyperelliptic} for two complex numbers $\alpha$ and $\tau$ with $\alpha^{-1}$ and $\tau$
sufficiently small --- note that $\alpha^{-1}$ and $\tau$ measure the size of the domain where the sewing elements 
$\hat{\psi}$ for the two families $\cF^{rat}_{2g +1}$ and $\cF^{hyp}_g$ converge. 
The parameter $u$ varies in $\cB$ as defined in \eqref{baseB}.
Then the loop $\xi \in \pi_1(\cB,i)$ which starts at $i$ and goes once around the points $-1$ and $0$ has monodromy of
infinite order.}

\end{rem}

\begin{rem}
{\em The previous argument, which proves the theorem for the Lie algebra $\slfr(2)$, fails when considering
the Lie algebras $\slfr(r)$ with $r > 2$. The main reason is the fact $\varpi^\dagger \not= \varpi$ for $ r > 2$,
where $\varpi$ is the first fundamental weight. }
\end{rem}

\section{Finiteness of the monodromy representation in genus one}

In this section we collect for the reader's convenience some existing results on the
monodromy representation on the conformal blocks associated to the Lie algebra $\slfr(2)$ for 
a family of one-marked elliptic curves labeled with the trivial representation. We consider the
upper half plane $\mathbb{H} = \{ \omega \in \mathbb{C} \ | \ \mathrm{Im} \  \omega  > 0 \}$ with the
standard action of the modular group $\PP \SL(2, \ZZ)$, which is generated by the two elements
$$ S =   \left(
\begin{array}{cc}
0 & -1  \\
1 & 0
\end{array}
 \right), \qquad 
 T =   \left(
\begin{array}{cc}
1 & 1  \\
0 & 1
\end{array}
 \right) \qquad \text{satisfying} \ S^2 = (ST)^3 = e.$$
Let $\cF$ denote the universal family of elliptic curves parameterized by $\mathbb{H}$. We denote by
$\cV^\dagger_{l,0}(\cF)$ the sheaf of conformal blocks of level $l$ with trivial 
representation at the origin. The sheaf $\cV^\dagger_{l,0}(\cF)$ has rank $l+1$ and for each $\lambda
\in P_l$ we obtain by the sewing procedure a section $\widehat{\psi}_\lambda$ over $\mathbb{H}$ given
by the formal series
$$ \langle \widehat{\psi}_\lambda(\omega) | \phi \rangle = \tau^{\Delta_\lambda} 
\sum_{d = 0}^\infty \langle \psi | \gamma_d \otimes \phi \rangle  \tau^d, \qquad \text{with} \  \ \tau = \exp(2i \pi \omega), $$
where $\psi$ is the unique (up to a multiplicative scalar) section of $\cV^\dagger_{l,\lambda,\lambda,0} (\PP^1)$
and $\phi$ is any element in $\cH_\lambda$.
Because of Theorem \ref{sewingisflat} the $l+1$ sections $\widehat{\psi}_\lambda$ are projectively flat for the
projective WZW connection on $\cV^\dagger_{l,0}(\cF)$ and are linearly independent by the factorization rules 
\eqref{factrules}.  Note that this decomposition of the sheaf $\cV^\dagger_{l,0}(\cF)$ into a sum of rank-$1$
subsheaves corresponds to a degeneration to the nodal elliptic curve given by $\mathrm{Im} \ \omega \rightarrow
\infty$, or equivalently $\tau = \exp(2i \pi \omega) \ra 0$. Moreover by 
evaluating the sections $\widehat{\psi}_\lambda$ at the highest weight vector $\phi = v_\lambda \in \cH_\lambda$
we obtain analytic functions
$\chi_\lambda(\omega) = \langle \widehat{\psi}_\lambda (\omega) | v_\lambda \rangle$, which correspond to the
character of the representations $\cH_\lambda$:
$$ \chi_\lambda(\omega) = \sum_{d=0}^\infty \sum_{i=1}^{m_d} (v_i(d)|v^i(d)) \tau^{\Delta_\lambda + d} = 
\sum_{d=0}^\infty \dim \cH_\lambda(d) \tau^{\Delta_\lambda + d} = \mathrm{tr}_{\cH_\lambda} (\tau^{L_0}),$$
see e.g. \cite{U} equation (4.3.1).  This shows that in the genus one case the local system given by the conformal blocks
with trivial marking equipped with the WZW connection coincides with the local system given by the characters
$\chi_\lambda(\omega)$. 
Moreover the monodromy action of the modular group $\PP \SL(2, \ZZ)$
on the vector space spanned by the characters $\{ \chi_\lambda(\omega) \}_{\lambda \in P_l}$ has been determined.

\begin{prop}[\cite{GW}]
The monodromy representation
$$ \rho_l: \PP \SL(2, \ZZ) \lra \PGL (l+1)$$
is given by the two unitary matrices $\rho(S)$ and $\rho(T)$
\begin{eqnarray*}
\rho(S)_{jk} &=& \sqrt{\frac{2}{l+2}} \sin \left( \frac{\pi jk}{l+2}   \right), \\
\rho(T)_{jk} & = & \delta_{jk} \exp\left(i\pi ( \frac{j^2}{2(l+2)} - \frac{1}{4}) \right).
\end{eqnarray*}
\end{prop}

With this notation the main statement of this section is the following

\begin{thm}
The image of the representation $\rho_l$ is finite.
\end{thm}

\begin{proof}
Using the explicit expression of the matrix $\rho_l(U)$ for any element $U \in \PP \SL(2, \ZZ)$ computed in
\cite{J} section 2, it is shown in \cite{G} section 2 that the matrix $\rho_l(U)$ has all its entries in the
set $\frac{1}{2(l+2)} \ZZ [ \exp(\frac{i\pi}{4(l+2)} )]$. Since moreover the representation $\rho_l$ 
is unitary, we may deduce finiteness along the same lines as in \cite{G} proof of corollary.
\end{proof}

\end{document}